\documentclass[a4paper,11pt]{article}
\usepackage[utf8]{inputenc}

\usepackage{verbatim}
\usepackage{amsfonts}
\usepackage{a4wide}
\usepackage{amsthm}
\usepackage{color}
\usepackage{graphicx}
\usepackage{amsmath,amssymb} 
\usepackage{lipsum}

\newcommand{\dd}{\mathrm d}

\renewcommand{\div}{\mathrm{div}\,}

\newcommand{\sgn}{\mathrm{sgn}\,}
\newcommand{\ess}{\mathrm{ess}\,}

\newcommand{\T}{\mathbb{T}}

\newcommand{\eps}{\varepsilon}
\newcommand{\Ll}{\mathcal L}
\newcommand{\J}{\mathcal J}
\newcommand{\R}{\mathcal R}
\renewcommand{\S}{\mathcal S}

\newcommand{\kkk}{\widetilde{\kappa}^k}
\newcommand{\nn}{\widetilde{\eta}}
\newcommand{\nnk}{\widetilde{\eta}^k}

\newcommand{\nnpk}{\widetilde{\eta}'^k}

\newtheorem{prop}{Proposition}
\newtheorem{thm}[prop]{Theorem}

\newtheorem{cor}[prop]{Corollary}
\theoremstyle{remark}
\newtheorem*{remark}{Remark}

\author{Micha{\l} {\L}asica \\
\small Institute of Applied Mathematics and Mechanics, University of Warsaw \\ 
\small Banacha 2, 02-097 Warszawa, Poland \\
\small \tt lasica@mimuw.edu.pl } 
\title{Analysis of solutions to a model parabolic equation with very singular diffusion} 
\date{\today}
\begin{document} 
 \maketitle 
 
 \begin{abstract}  
 \noindent
 We consider a singular parabolic equation of form 
 \[
  u_t = u_{xx} + \tfrac{\alpha}{2}(\sgn u_x)_x 
 \]
 with periodic boundary conditions. Solutions to this kind of equations exhibit competition between smoothing due to one-dimensional Laplace operator and tendency to create flat facets due to strongly nonlinear operator $(\sgn u_x)_x$ coming from the total variation flow. We present results concerning analysis of qualitative behaviour and regularity of the solutions. Our main result states that locally (between moments when facets merge), the evolution is described by a system of free boundary problems for $u$ in intervals between facets coupled with equations of evolution of facets. In particular, we provide a proper law governing evolution of endpoints of facets. This leads to local smoothness of the motion of endpoints and the unfaceted part of the solution.
\end{abstract}
\smallskip
\begin{center}
\end{center}
 \section{Introduction} 
 We consider the following equation 
\begin{equation} 
 \label{theequation} 
 u_t = u_{xx} + \tfrac{\alpha}{2} (\sgn u_x)_x 
\end{equation}
for a function $u = u(t,x)$ defined on a time-space band $[0,T] \times \T$, with $\T$ denoting the $1$-dimensional torus that will be represented by interval $[0,1]$ with periodic boundary conditions. In \eqref{theequation}, $\alpha$ is a positive constant coefficient. 

The differential operator on the right hand side of \eqref{theequation} may be written in divergence form 
\begin{equation} 
 \label{divergenceform} 
 L(u_x)_x = \left(\left(1 + \tfrac{\alpha}{2} \tfrac{1}{|u_x|}\right)u_x\right)_x  
\end{equation}
and thus we see that \eqref{theequation} may be seen as a one-dimensional nonlinear heat flow with diffusivity singular whenever $u_x = 0$. The basic properties of such equations (such as existence of solutions) are founded by nonlinear semigruop theory of Y.\;K{\=o}mura \cite{komura}. However, this particular form of singularity (with diffusivity of order $|u_x|^{-1}$) results in specific notions of regularity and a characteristic qualitative property of solutions -- their graphs immediately develop flat parts (\emph{facets}) whose evolution involves translation in normal direction. For this reasons, there are several intertwined strands of literature that deal with deeper analysis of this type of equations. 

One of them is concerned with singular anisotropic mean curvature (AMC) flows of surfaces that appear in models of crystal growth. The analysis of stationary points of these flows (with e.\,g.\;volume constraints) began with the construction of G.\;Wullf \cite{wulff}, whose mathematical validity was further elaborated \cite{taylor1, fm, palmer, morgan, anl}. The investigation of actual flows began in early nineties with works of J.\;Taylor \cite{taylor2} and S.\;Angenent with M.\;Gurtin \cite{ag1}, who constructed evolutions of suitable polygonal curves in the plane under \emph{crystalline} curvature flows (which locally correspond to parabolic equations with diffusivity vanishing outside of a finite set of values of $u_x$, where it is singular). Later, F.\;Almgren, J.\;Taylor and L.\;Wang \cite{atw} provided, in the language of integral currents, a much more general construction of \emph{flat flow} in arbitrary dimension that includes as special cases the crystalline flow of polygons \cite{at} and smooth AMC flows. 

Singular AMC flows were also studied by a group around Y.\;Giga, who in particular with T.\;Fukui investigated semigroup solutions to singular equations in divergence form, and in this setting proved that the facets translate with speed inversely proportional to their length (which amounted to characterising the minimal selection of the subdifferential of the underlying functional) \cite{fg}. A more involved qualitative analysis of semigroup solutions was performed by K.\;Kielak, P.\;Mucha and P.\;Rybka \cite{kielak} who identified them as pointwisely defined \emph{almost classical solutions}. In order to treat also equations in non-divergence form (such as AMC flow of graphs), Y.\;Giga and M.-H.\;Giga extended the notion of viscosity solutions developed for regular anisotropy in \cite{cgg} to the singular case, first for graphs \cite{gg} and then for closed curves \cite{gg2}. This notion generalises regular evolutions of Taylor in a more discerning way, admits arbitrary continuous data and also provides uniqueness in contrast to the flat flow. 

Meanwhile, interest arose in the total variation (TV) flow $u_t = \div \left(\frac{\nabla u}{|\nabla u|}\right)$ due to connection with image denosising algorithms \cite{rof}. F.\;Andreu, C.\;Ballester, J.\;Mazon and V.\;Caselles studied semigroup solutions to the TV flow in arbitrary dimension and provided a characterisation of subdifferential of the underlying functional \cite{mtv, mazonbook}. This results were transferred to the multidimensional AMC flows setting by G.\;Bellettini, M.\;Novaga and others, who provided a notion of regular solution (requiring that it admits suitably regular selection of the ``anisotropic normal'' field) and obtained existence in some cases \cite{bnp1, bnp2, bccn1, bccn2}. 

The mentioned results consider either general theory (striving to admit smooth as well as crystalline case) or concentrate on the qualitative properties of the crystalline case (with $L$ piecewise constant) in particular. On the other hand, in our case $L$ is strictly monotone. We note that from the modeling viewpoint, this would correspond to crystals (lumps) of metal -- in this case the optimal shape still exhibits facets, but also has smoothly rounded edges (see e.\,g.\;\cite{cusp}). Inhomogeneous systems with both the Laplacian and a very singular operator were also investigated in relation to potential application in restoration of images \cite{muszkieta}. From the viewpoint of pure mathematics, \eqref{theequation} displays competition between standard diffusion operator $u_{xx}$ which tends to smoothen solutions and strong directional diffusion operator $(\sgn u_x)_x$ that tends to create facets. 

The equation \eqref{theequation} was investigated by Mucha and Rybka, who collected several basic observations concerning the behaviour of its solutions in \cite{sdd}. In particular, they obtained some regularity results in the language of Sobolev spaces and noticed that in any moment of time $t>0$ the solutions do not allow isolated extremal points (which are immediately turned into facets of finite length) nor facets embedded in monotone graph (which are immediately destroyed). Furthermore, they sought to analyse fine behaviour of endpoints of facets. For this purpose, they considered solutions to \eqref{theequation} for a certain class of initial data and provided a condition deciding whether a facet will grow or shrink. However, as we will see, their initial data were not regular (in the sense appropriate to \eqref{theequation}), as the one-sided second derivative on the facet endpoint was not equal to the ``crystalline curvature'' of the facet. Finally, we mention that equation \eqref{theequation} is also the subject of recent works of P.\;Mucha \cite{pmscb}, who showed example of facet breaking in the forced case, and of T.\;Asai with P.\;Rybka \cite{asai}, who proved that the number of facets is a non-increasing function of time. 

The content of the present paper is following. In Section 2 we collect information about existence and global regularity of semigroup
solutions obtaining 

\begin{thm} 
\label{basicthm}
Given $u_0 \in L^2(\T)$ and any $T>0$, there exists a unique solution 
$$u \in C([0,T];L^2(\T)) \cap L^2(0,T; H^1(\T))$$ 
to \eqref{theequation} with initial datum $u_0$. The solution becomes instantly regularised so that for any $\delta > 0$, 
$$\kappa \equiv u_t \in L^\infty(\delta, T; L^2(\T)) \cap L^2(\delta, T; H^1(\T))$$ 
(though typically $\kappa \notin C(\delta, T; L^2(\T))$). 

In every moment of time $t>0$ there exists a subdivision 
of $\T$ into a finite number of intervals $F^k(t)$ and $I^k$. In each $F^k(t)$ the graph of solution consists of a single facet (i.\,e.\;$u$ is constant). In each $I^k(t)$ the solution is monotone, furthermore $u \in H^3(I^k(t))$ for a.\,e.\;$t>0$. The speed of vertical motion of facets is given by the ``crystalline curvature'' whose absolute value is equal to $\frac{\alpha}{|F^k(t)|}$.  

After a time $T^* \leq \sqrt 2 \|u_0\|_{L^2(\T)}$ the solution becomes constant and equal to $\int_\T u_0$. 
\end{thm}

Let us underline that crucial role is played by the quantity $u_t$ which we will persistently denote $\kappa$, as it corresponds to the anisotropic curvature of AMC flows. It is more or less evident that $\kappa$ is equal to $u_{xx}$ in unfaceted regions. However, in contrast to $u_{xx}$, $\kappa$ belongs to $H^1(\T)$ in a.\,e.\;point of time. In faceted parts, $\kappa$ is nonlocal and equal to the already mentioned crystalline curvature. 

In Section 3 we investigate the behaviour of the solution between the instances of time when facets merge. In these time intervals, the solution can be described by a system of free boundary problems for evolution of $u$ in $I^k(t)$ and evolution of intervals $I^k(t)$ themselves. In order to pose this system correctly, it is essential to provide valid law of motion of endpoints of facets (equivalently, intervals $I^k$). Our construction shows that the proper formula for the speed of horizontal motion of endpoint $z=z(t)$ of $I^k$ adjacent to $F^j$ is 
$$\dot z = \pm \frac{|F^j|}{\alpha} \kappa_x$$
which is well defined a.\,e.\;due to regularity of $\kappa$. 

The structure of the free boundary system is similar to the one of Stefan problems, for which (and whose generalisations) extensive theory is available (see e.\,g\;\cite{fp, fpnonlinear}). However, due to coupling present in our problem we cannot simply apply known results and instead we provide our own theorem on local existence and smoothness. 

\begin{thm} 
\label{regularthm}
For almost every time instance $r>0$ there exists $s>r$ such that the number of facets $n$ is constant in $]r,s[$ and  
$$I^k \in C^\infty(]r,s[)^2 \quad \text{and} \quad u \in C^\infty(I^k_{r,s})$$ 
for each $k=1,\ldots,n$, where 
\begin{equation} 
\label{ikrs}
I^k_{r,s}=\bigcup_{t \in ]r, s[} \{t\} \times \overline{I^k(t)}. 
\end{equation}  
\end{thm}

It remains an open question whether it can be extended to whole intervals between facets merging. For one-phase Stefan-like problems a singularity of type $\lim_{t \to t^*} |z(t)| = \infty$ can occur depending on specific structure of the system and initial datum \cite{fp, fpnew}. Note that for bad enough initial data the solution may not exist in any time interval \cite{fpnew}.

 \section{Basic properties of solutions} 
 Formally, the equation \eqref{theequation} may be viewed as a parabolic inclusion 
\begin{equation} 
 \label{inclusion} 
 u_t \in \Ll u
\end{equation}
in the sense of $H^{-1}(\T)$ with $\Ll u = L(u_x)_x$, where $L$ is treated as a maximal monotone graph 
\begin{equation} 
 \label{Lgraph} 
 L(p) = \left\{ \begin{array}{cr} 
                  p - \tfrac{\alpha}{2} & \text{if } p < 0, \\
                  {}[-\tfrac{\alpha}{2},\tfrac{\alpha}{2}] & \text{if } p = 0, \\ 
                  p + \tfrac{\alpha}{2} & \text{if } p > 0. \\
                 \end{array}
                 \right. 
\end{equation}
The multifunction $L$ is the subdifferential of $J(p) = \tfrac{1}{2}(p^2 + \alpha |p|)$. Thus, the operator $\Ll$ may be defined as negative of the subdifferential of a functional $\J$ defined on $L^2(\T)$ by
\begin{equation} 
 \label{J} 
 \J(u) = \tfrac{1}{2}\int_\T u_x^2 +  \alpha |u_x|
\end{equation}
whenever $u \in H^1(\T)$ and $\J(u) = + \infty$ otherwise. Clearly, $D (\J) = H^1(\T)$ and $\J$ is an equivalent norm on $H^1(\T)$. Furthermore, $\J$ is convex and lower semicontinuous (in particular, if $(u_n) \subset D(\J)$ converges to $u \in L^2(\T)\setminus D(\J)$, then $J(u_n) \to \infty$). Let us now calculate formally the subdifferential $\partial \J$. 

\begin{prop} 
 \label{subdifferential}
 We have 
 \[
  D(\partial \J) = \left\{ 
  \begin{aligned} 
  &u \in H^1(\T) \text{ such that there exists }\\ 
  &\text{ a selection } \sigma \in H^1(\T) \text{ satisfying } \sigma \in L(u_x) \text{ in } \T 
  \end{aligned} 
  \right\} 
 \]
 and 
 $$\partial \J(u) = \{- \sigma_x \colon \sigma \in H^1(\T), \sigma \in L(u_x) \text{ in } \T\} .$$
\end{prop} 

\begin{proof} 
 Let $u \in D(\J) = H^1(\T)$. Whenever $w \in \partial \J(u)$, we have 
 \begin{equation}
 \label{defsubd}
 \J(u + \varphi) \geq \J(u) + (w, \varphi) 
 \end{equation} 
 for any $\varphi \in L^2(\T)$, with $(\,\cdot\, , \cdot\,)$ denoting the standard scalar product in $L^2(\T)$. Clearly, it is sufficient to consider  $\varphi$ of form $\varphi = \lambda \psi$ with $\psi \in H^1(\T)$, $\lambda > 0$. Then \eqref{defsubd} becomes 
 \begin{equation} 
  \tfrac{1}{2}\int_\T |u_x + \lambda \psi_x|^2 +  \alpha |u_x + \lambda \psi_x| - \tfrac{1}{2}\int_\T |u_x|^2 +  \alpha |u_x| \geq \lambda (w, \psi)
 \end{equation}
which we transform and divide by $\lambda$ to obtain 
\begin{equation} 
 \label{calc1}
 \tfrac{1}{2}\int_\T \lambda \psi_x^2 + 2 u_x \psi_x  + \tfrac{\alpha}{2} \int_{\{u_x = 0\}} |\psi_x| + \tfrac{\alpha}{2} \int_{\{u_x \neq 0\}} \tfrac{1}{\lambda} (|u_x + \lambda \psi_x| - |u_x|) \geq (w, \psi) .
\end{equation}
Next, we pass to the limit $\lambda \to 0^+$. In the limit, the first term of the l.\,h.\,s.\;vanishes. To treat the last one, we notice 
\begin{multline} 
 \int_{\{u_x \neq 0\}} \tfrac{1}{\lambda} (|u_x + \lambda \psi_x| - |u_x|) \\ = \int_{\{0< |u_x| \leq \lambda |\psi_x|\}} \tfrac{1}{\lambda} (|u_x + \lambda \psi_x| - |u_x|) + \int_{\{\lambda |\psi_x| < |u_x|\}} (\sgn u_x) \psi_x    
\end{multline}
and 
\begin{equation} 
 \int_{\{\lambda |\psi_x| < |u_x|\}} (\sgn u_x) \psi_x \to \int_{\{u_x \neq 0\}} (\sgn u_x) \psi_x,
\end{equation}
\begin{equation} 
\left|\int_{0< |u_x| \leq \lambda |\psi_x|} \tfrac{1}{\lambda} (|u_x + \lambda \psi_x| - |u_x|)\right| \leq 2 \int_{0< |u_x| \leq \lambda |\psi_x|} |\psi_x| \to 0 
\end{equation} 
as $\lambda \to 0^+$ by virtue of dominated converegence. Therefore, we obtain 
that if $w$ belongs to $ \partial\J(u)$, the inequality  
\begin{equation} 
\label{subequiv}
\int_\T u_x \psi_x  + \tfrac{\alpha}{2} \int_{\{u_x = 0\}} |\psi_x| + \tfrac{\alpha}{2} \int_{\{u_x \neq 0\}} (\sgn u_x) \psi_x \geq (w, \psi)
\end{equation}
is satisfied for each $\psi \in H^1(\mathbb T)$. The converse is also true, as 
\eqref{subequiv} implies \eqref{calc1}. Thus, if $w = - \sigma_x$, where $\sigma 
\in H^1$ is a selection of the multifunction $(1 + \tfrac{\alpha}{2}\sgn)\circ 
u_x$, then $w \in \partial \J(u)$. 

On the other hand, take any $w$ that 
satisfies \eqref{subequiv}. Taking $\psi \equiv 1$ we see that $\int_{\mathbb T} 
w = 0$, and thus $w$ admits a primitive (defined up to a constant, which we 
will choose in a moment), i.\,e.\;$w = - \sigma_x$ for a function $\sigma \in 
H^1(\mathbb T)$. Now, take any $\psi \in H^1(\mathbb T)$ (note that $\{\psi_x 
\colon \psi \in H^1(\mathbb T)\} = \dot L^2(\mathbb T)$) such that $\psi_x = 0$ 
on $\{u_x = 0\}$. Considering both $\psi$ and $-\psi$ in \eqref{subequiv} 
yields 
\begin{equation} 
 \int_{\{u_x \neq 0\}} (u_x + \sgn u_x) \psi_x = \int_{\{u_x \neq 0\}} \sigma 
\psi_x
\end{equation}
and thus, $\sigma = u_x + \sgn u_x$ a.\,e.\;in $\{u_x \neq 0\}$ up to a 
constant (which we now choose to be $0$). 

Finally, take any $\psi \in H^1(\mathbb T)$ such that $\psi_x = 0$ 
on $\{u_x \neq 0\}$. Considering $\psi$ and $-\psi$ in \eqref{subequiv} yields 
\begin{equation} 
 \left| \int_{\{u_x = 0\}} \sigma \psi_x \right| \leq \int_{\{u_x = 0\}} 
|\psi_x|
\end{equation}
which implies that $\sigma = c + \sigma^*$ a.\,e.\;in $\{u_x = 0\}$ with 
$\|\sigma^*\|_{L^\infty(\{u_x = 0\})} \leq \frac{\alpha}{2}$ and $c \in 
\mathbb R$. Unless $u$ is constant in $\mathbb T$, our previous choice of 
$\sigma$ together with its regularity imply that $c=0$. If $u$ is constant, we 
choose $c=0$.  
\end{proof}

\begin{remark} 
 We have $D(\partial \J) \subset H^2(\T) \subset C^{1+ \frac{1}{2}}(\T)$. 
Indeed, note that for any $u \in D(\partial \J)$, $u_x$ is representable as the 
composition of the piecewise linear continuous function $p \mapsto \sgn p (|p| - 
1)_+$ and $\sigma \in H^1(\T)$ s.\,t.\;$- \sigma_x \in \partial \J(u)$. In 
particular, the function $u_{xx} = \sigma_x \mathbf 1_{\{u_x \neq 0\}}$ (defined 
independently of $-\sigma_x \in \partial \J(u)$) is the distrubutional second 
derivative of $u$ and belongs to $L^2(\T)$. Furthermore, it is an easy 
observation that $D(\partial \J)$ is dense in $L^2(\T)$. 
\end{remark}

Equipped with the above observations concerning $\Ll \equiv - \partial \J$, we may use semigroup theory to obtain basic existence and regularity result for the inclusion \eqref{inclusion} \cite[Chapter IV, Theorems 2.1 and 2.2]{barbu}. 
\begin{prop} 
 \label{existence}
 Let $u_0 \in L^2(\T)$. The problem \eqref{inclusion} with initial condition $u_0$ has a unique solution 
 $$u \in C([0,T];L^2(\T)) \cap L^2(0,T; H^1(\T))$$ 
 which satisfies 
 $$u_t \in L^\infty(\delta, T; L^2(\T)) \text{ for every } 0 < \delta < T, $$ 
 $$u(t) \in D(\Ll) \text{ for all } t \in ]0,T[.$$
 
 Moreover, we have 
 $$\frac{\dd^{\scriptscriptstyle{+}}}{\dd t} u = \Ll^0 u \text{ for all } t \in ]0,T[.$$
 Here, $\frac{\dd^{\scriptscriptstyle{+}}}{\dd t} u$ denotes the right-sided time derivative of $u$ and $\Ll^0$ is the minimal selection of $\Ll$, i.\,e.\;for $u \in D(\Ll)$, $\Ll^0 u$ is the (uniquely defined) element of $\Ll u$ of minimal norm in $L^2(\T)$. 
\end{prop}

The remark after Proposition \ref{subdifferential} states that the regularity properties of $\Ll$ are, in a sense, at least as good as those of the (one-dimensional) Laplace operator. However, the dissipation in $\Ll$ is essentialy stronger than that of $\Delta$, so higher regularity could be expected. The following proposition (in a way, a corollary of Proposition \ref{subdifferential}) captures this additional regularity. Roughly, it states that $u \in D(\Ll)$ if and only if $u \in H^2(\T)$ and $\T$ may be divided into a finite number of (non-degenerate) intervals where $u$ is constant and intervals where $u$ is monotone. 

\begin{prop} 
\label{decomp}
 Let $u \in D(\Ll)$. Then, there exists a disjoint decomposition of $\T$ into a number $n$ of (non-degenerate) open intervals $I^k = ]a^k, b^k[ \subset \T$, $k=1,\ldots, n$ and $n$ (non-degenerate) closed intervals $F^k = [b^{k-1}, a^k]  \subset \T$, $k=1,\ldots,n$ with $b_0 \equiv b_n \mod 1$ such that 
 \begin{itemize} 
  \item[(i)] $u_x = 0$ in each $F^k$, $F^k$ is a maximal closed interval with this property and $u$ attains (improper) local extremum in $F^k$,    
  \item[(ii)] $u$ is monotone in each $I^k$, 
  \item[(iii)] $|F^k| \geq n \frac{\alpha^2}{E(u)^2}$, where $E(u) = \min\limits_{\sigma_x \in \Ll u} \|\sigma_x\|_{L^2(\T)}$. 
 \end{itemize}
 On the other hand, if $u \in H^2(\T)$ and a finite decomposition $\{I^k, F^k\}$ of $\T$ satisfies conditions (i, ii), then $u \in D(\Ll)$ and (iii) holds. 
 
 Furthermore, $\Ll^0 u |_{I^k} = u_{xx}$ in each $I^k$, $\Ll^0 u |_{F^k} = - \frac{\alpha}{|F^k|}$ if $u$ attains an improper maximum in $F^k$ and $\Ll^0 u |_{F^k} = \frac{\alpha}{|F^k|}$ in the other case. 
\end{prop}
\begin{figure}
\begin{center}
\includegraphics[scale=0.5]{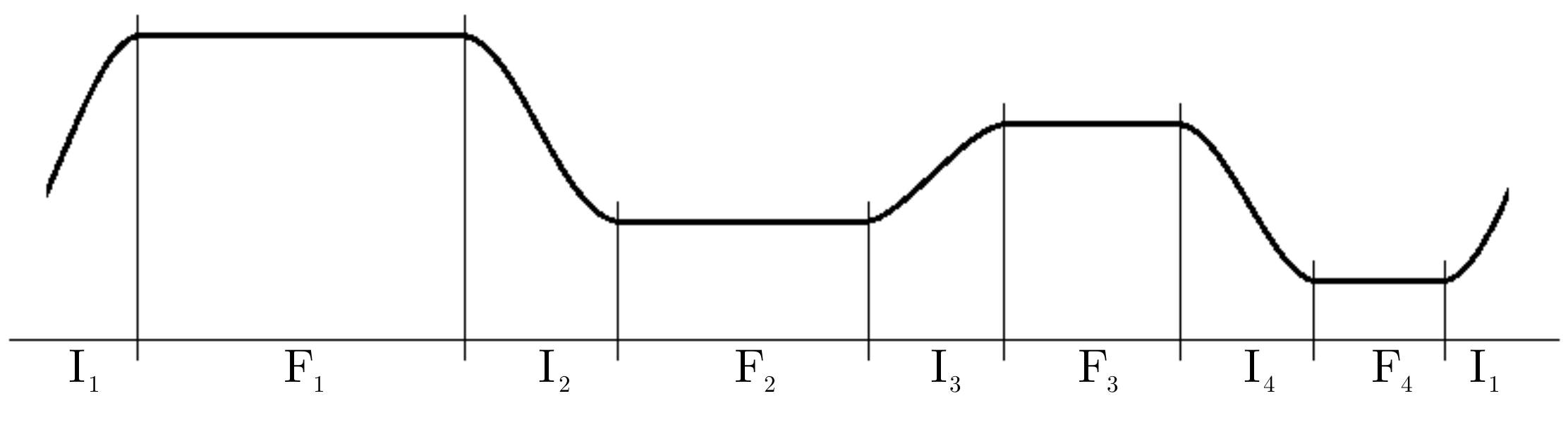}
\caption{Graph of a typical $u \in D(\Ll)$.}
\end{center}
\label{graph}
\end{figure} 
\begin{proof} 
 The existence of a (possibly infinite) decomposition of $\T$ satisfying properties \emph{(i, ii)} is an obvious consequence of continuity of $u_x$. Finiteness follows from property \emph{(iii)}. To prove property \emph{(iii)}, we observe that for any $F^k$ and any $\sigma_x \in \Ll u$ we have 
 \begin{equation} 
  \label{calcdec}
  \int_{F^k} \sigma_x^2 \geq \int_{F^k} \left(\frac{\sigma(a^k) - \sigma(b^{k-1})}{|F^k|}\right)^2 = \frac{\alpha^2}{|F^k|} . 
 \end{equation}
 Indeed, the inequality in \eqref{calcdec} is a consequence of the fact that the affine function minimizes the functional $\int_a^b u_x^2$ on $H^1$ with prescribed boundary values. The equality follows from continuity of $\sigma$ and property \emph{(iii)} of the decomposition, as we necessarily have 
 $$\lim_{x \to (b^{k-1})^-} \sigma(x) = \lim_{p \to 0^\pm} L(p) = \pm \tfrac{\alpha}{2}, \quad \lim_{x \to (a^{k})^+} \sigma(x) = \lim_{p \to 0^\mp} L(p) = \mp \tfrac{\alpha}{2} .$$ 
 
 Now, assume that $u \in H^2(\T)$ and a finite decomposition $\{I^k, F^k\}$ of $\T$ satisfying conditions \emph{(i, ii)} exists. Then we define $\sigma^0(x)$ as $L(u_x(x))$ whenever $u_x(x) \neq 0$. Next, we consider the case that $u_x(x) = 0$. If $x \in \overline{I^k}$ and $u$ is non-decreasing (resp.\;non-increasing) in $I^k$, we put $\sigma^0(x)=\frac{\alpha}{2}$ (resp.\;$\sigma^0(x)=-\frac{\alpha}{2}$). We are left with the task of defining $\sigma^0$ in the interior of intervals $F^k$. As we have already defined $\sigma^0$ in each $a^k$ and $b^k$, we extend it continuously to $F^k$ by suitable affine functions. The function $\sigma^0$ we obtained belongs to $H^1(\T)$. 
 
 Finally, note that due to \eqref{calcdec} and maximality of $F^k$, 
$\sigma^0_x$ necessarily minimizes the $L^2(\T)$ norm among elements of 
$D(\Ll)$. Thus, $\Ll^0 u = \sigma^0_x$. 
\end{proof}

Formally, we may write 
\begin{equation} 
 \label{formally}
 u_{tt} = L(u_x)_{xt} = (L'(u_x) u_{tx})_x .
\end{equation} 
As $L' > 0$ in $\mathcal D'(\mathbb R)$, we could expect \eqref{formally} to yield additional regularity of solutions to \eqref{inclusion}, but due to lack of proper definition of the term $L'(u_x)$ in \eqref{formally} we need to proceed by approximation. Hence, let us denote by $J_\eps$ smoothened versions of $J$ given by 
$$J_\eps(p) = \tfrac{1}{2} p^2 + \alpha (\eps + p^2)^{\frac{1}{2}} $$
and by $L_\eps$ its derivative 
$$L_\eps(p) = J_\eps'(p) = p + \alpha \frac{p}{(\eps + p^2)^{\frac{1}{2}}}.$$ 
In particular we have 
$$1 \leq L_\eps'(p)  = 1 + \frac{\alpha \eps}{(\eps + p^2)^{\frac{3}{2}}} \leq 1 + \frac{\alpha}{\eps^\frac{1}{2}}.$$ 

Analysing the approximate problem 
\begin{equation} 
 \label{approx} 
 u^\eps_t = L_\eps (u^\eps_x)u^\eps_x \quad \text{in }\T 
\end{equation}
we obtain the following result. 
\begin{prop} 
 \label{curvature} 
 Let $u$ be the unique solution to \eqref{inclusion} with $u_0 \in L^2(\T)$. Then, for any $\delta > 0$ we have 
 $$\kappa \equiv u_t \in L^2(\delta,T; H^1(\T)) \cap L^\infty(\delta,T; L^2(\T)) .$$
\end{prop}
\begin{proof} 
 Using either the semigroup theory \cite{barbu} or fixed point methods \cite{lsu} we obtain the existence of weak solutions to \eqref{approx} in $C([0,T];L^2(\T)) \cap L^2(0,T; H^1(\T)) \cap H^1(\delta, T; L^2(\T)) \cap L^2(\delta, T; H^2(\T))$ for any $\delta > 0$. The time derivative of approximation $\kappa^\eps = u^\eps_t$ satisfies formally
 \begin{equation} 
  \label{timeder} 
  \kappa^\eps_t = (L_\eps'(u^\eps_x) \kappa^\eps_x)_x \quad \text{in } \T, 
 \end{equation}
 Thus, as $L_\eps'(u^\eps_x)$ is uniformly positive and bounded in $]0,T[ \times \T$ for any given $\eps >0$, we may solve the problem \eqref{timeder} with initial datum cut off. Using e.\,g.\;\cite[Chapter III, Proposition 4.1]{showalter}, we get unique solution $\kappa^\eps$ in the class $C([\delta,T];L^2(\T)) \cap L^2(\delta,T; H^1(\T))$ which clearly coincides with $u^\eps_t$. Testing the problem with the solution $\kappa^\eps$ we obtain following estimate independent of $\eps$
 $$\tfrac{1}{2} \ess\!\!\!\sup_{t\in [\delta,T]} \|\kappa^\eps(t,\cdot\,)\|_{L^2(\T)}^2 + \|\kappa^\eps_x\|_{L^2(\delta,T;L^2(\T))}^2 \leq C(\delta) .$$
 As $\kappa^\eps \to \kappa$ in $\mathcal D'(]\delta,T[\times \T)$, we arrive at the assertion. 
\end{proof}
\begin{cor} 
 Propositions \ref{existence} and \ref{curvature} imply that $u_x$ belongs to $H^1(\delta,T, L^2(I))\times L^2(\delta, T, H^2(I))$, and therefore also to the parabolic H\" older space $C^{\frac{1}{4}, \frac{1}{2}}([\delta,T]\times \T)$ \cite[Chapter II, Lemma 3.3]{lsu}. 
\end{cor}

\begin{remark} 
 If \eqref{formally} was a regular parabolic equation, one would be able to obtain $\kappa$ at least in $C([\delta,T];L^2(\T))$. The reasoning in the proof of Proposition \ref{curvature} does not lead to such regularity, as the required estimate on $L^2(\delta,T;H^{-1}(\T))$ norm of $\kappa^\eps$ does not hold. 
\end{remark}

\begin{prop} 
\label{asymptotic} 
 The solution $u$ becomes constant and equal to $\int_\T u_0$ after time $T^*>0$ such that $T^* \leq \sqrt 2 \left\|u_0 - \int_\T u_0\right\|_{L^2(\T)}$.   
\end{prop}
\begin{proof} 
 Assume first that $\int_\T u_0$ and consequently $\int_\T u = 0$ in a.\,e.\;time instance. Testing the problem \eqref{inclusion} with $u$ we obtain 
 \begin{equation} 
 \tfrac{1}{2}\left(\int_\T u^2\right)_t + \int_\T (u_x^2 + |u_x|) = 0 
 \end{equation} 
 in almost all instances of time. As 
 \begin{equation} 
  \int_\T u^2 \leq 2 \left( \int_\T |u_x|\right)^2, 
 \end{equation}
 this yields 
 \begin{equation} 
 \label{asympest1}
 \tfrac{1}{2}\left(\|u\|_{L^2(\T)}^2\right)_t + \tfrac{1}{\sqrt 2} \|u\|_{L^2(\T)} \leq 0.  
 \end{equation} 
 As long as $u \neq 0$ we may divide \eqref{asympest1} by $\|u\|_{L^2(\T)}$ obtaining 
 \begin{equation} 
  \left(\|u\|_{L^2(\T)}\right)_t \leq - \tfrac{1}{\sqrt 2} . 
 \end{equation}
 Integrating over time we see that $u=0$ in $t = \sqrt 2 \|u_0\|_{L^2(\T)}$ (and afterwards). 
 
 Finally, let us relax the assumption of vanishing mean of $u$. It suffices to notice that $u - \int_\T u_0$ is the solution to \eqref{inclusion} with initial datum $u_0 - \int_\T u_0$. 
\end{proof} 

Propositions \ref{subdifferential}-\ref{curvature} comprise the proof of Theorem \ref{basicthm}. \qed

 \section{Characterisation of regular evolutions} 
 In \cite{asai} it is proved that, starting from regular datum (which in particular admits only finite number of facets), the number of facets of a solution to \eqref{theequation} is a non-increasing function of time. On the other hand, we start from datum in $L^2(\T)$ which is however instantly regularised. Thus, the number of facets is also a non-increasing, though possibly unbounded on $]0,T]$, function of time and there is a countable number of moments of merging. Let now $r,s$ denote any subsequent two of those. In $]r,s[$ the number of facets is constant and we may postulate that there are well-defined functions $I^k = I^k(t) = ]a^k(t), b^k(t)[$, $F^k = F^k(t) = [b^{k-1}(t), a^k(t)]$, $k=1,\ldots,n$. Taking into account Proposition \ref{decomp} we note that the existence of sufficiently regular solutions to \eqref{theequation} in $]r,s[$ is equivalent to the existence of solutions to the following system of free boundary problems 
\begin{equation} 
 \label{free}
 u_t = u_{xx} \qquad \text{in } I^k(t),  
\end{equation}
\begin{equation}
 \label{freebc1}
 u_x = 0 \qquad \text{in } \partial F^k(t),  
\end{equation}
\begin{equation} 
 \label{freebc2} 
 u_t = (-1)^k \tfrac{\alpha}{|F^k|} \qquad \text{in } \partial F^k(t) 
\end{equation} 
in $]r,s[$ for $k=1,\ldots,n$. 
In order to solve (\ref{free}-\ref{freebc2}), we consider differentiated system for $\kappa = u_t$ and $I^k=I^k(t)$
\begin{equation}
 \label{freek}
 \kappa_t = \kappa_{xx} \qquad \text{in } I^k(t),
\end{equation}
\begin{equation} 
  \label{freekbc1} 
  \kappa = (-1)^k \tfrac{\alpha}{|F^k|} \qquad \text{in } \partial F^k(t), 
\end{equation}
\begin{equation} 
 \label{freekbc2} 
 - \dot z = (-1)^k \tfrac{|F^k|}{\alpha} \kappa_x  \qquad \text{for } z \in \partial F^k,
\end{equation} 
$k=1, \ldots, n$. Equation \eqref{freekbc1} is rewritten \eqref{freebc2}, while \eqref{freekbc2} follows by differentiation of \eqref{freebc1} with respect to time, yielding 
\begin{equation*}
 u_{xt}(\,\cdot\,, z) + u_{xx}(\,\cdot\,, z) \dot z = 0,
\end{equation*}  
and application of (\ref{freebc2},\ref{free}). We will solve (\ref{freek}-\ref{freekbc2}) locally in a time interval denoted for simplicity by $]0, t_*[$ given a regular initial datum. 

\begin{prop} 
\label{freeexistence}
 Let $(I^k_0=]a^k_0, b^k_0[ \colon k=1, \ldots,n)$ be a collection of disjoint intervals in $\T$ ordered by succession. Let $\kappa_0 \in H^1\left(\bigcup_{k=1,\ldots,n} I^k_0\right)$ satisfy compatibility condition
 $$\kappa_0 = (-1)^k \tfrac{\alpha}{|F^k_0|} \qquad \text{in } \partial F^k_0,$$ 
 where we denoted $F_{k-1}^0 = [b^{k-1}_0, a^k_0]$, $k=1,\ldots,n$, $b^0_0 = b^n_0$. 
 Assuming that $t_*$ is small enough, there exists a unique solution $(\kappa, (I^k))$ to (\ref{free}-\ref{freebc2}) in $]0, t_*[$ satisfying $\kappa(0, \cdot) = \kappa_0$ and
 \[ 
  \|\kappa_{xx}(t,\cdot\,) \|_{L^2(I^k)} \in L^2(0,t_*) , \quad \|\kappa_x(t,\cdot\,) \|_{L^2(I^k)} \in L^\infty(0,t_*), 
 \]  
 \[ 
  I^k \in H^1(0,t_*)^2 
 \]
 for all $k=1, \ldots,n$. 
\end{prop}

\begin{proof} 
We rescale each $\kappa|_{I^k}$ to a fixed interval $I=[0,1]$ with homogeneous Dirichlet boundary conditions, namely we introduce $\kkk$ defined by
\begin{equation}
 \label{ktilde} 
 \kkk(t,x) = \kappa(t, \Phi_{a,b}^k (t, x)) - f_{a,b}^k(t,x),
\end{equation}
for all $(t,x) \in [0,t_*]\times I$, $k=1,\ldots,n$. Here, $\Phi_{a,b}^k(t,\,\cdot\,) \colon [0,1] \to I^k(t)$ denotes the affine bijection
\begin{equation} 
 \Phi_{a,b}^k(t,x) = a^k(t) + |I^k(t)| x 
\end{equation}
and $f_{a,b}^k(t,\,\cdot\,)$ is the affine function given by
\begin{equation}  
 f_{a,b}^k(t,x) = (1-x)(-1)^{k-1} \tfrac{\alpha}{|F_{k-1}(t)|} + x (-1)^k \tfrac{\alpha}{|F_{k}(t)|} .
\end{equation}
Note that if $(a,b)$ is continuous, the rescaling $(\kkk) \leftrightarrow \kappa$ is a bi-Lipschitz mapping between $L^p(0,t_*, H^l(I))^n$ and its non-cylindrical equivalent in the image for any $p \in [1, \infty], l =0, 1, \ldots$. Functions $\kkk$ are expected to satisfy equations 
\begin{equation}
 \label{kkkeqn}
 \kkk_t = \tfrac{1}{|I^k|^2} \kkk_{xx} +  \tfrac{1}{|I^k|} \Phi_{a,b,t}^k \cdot (\kkk_x + f_{a,b,x}^k) - f_{a,b,t}^k \quad \text{in } I,
\end{equation} 
\begin{equation} 
 \label{kkkbc1}
 \kkk = 0 \quad \text{on } \partial I, 
\end{equation}
\begin{equation} 
\begin{aligned}
 \label{kkkbc2}
 - \dot a^k = &\tfrac{1}{\alpha} (-1)^{k-1} \tfrac{|F_{k-1}|}{|I^k|}(\kkk_x(\,\cdot\,,0) + f_{a,b,x}^k(\,\cdot\,,0)), \\
 - \dot b^k = &\tfrac{1}{\alpha} (-1)^{k} \tfrac{|F_{k}|}{|I^k|}(\kkk_x(\,\cdot\,,1) + f_{a,b,x}^k(\,\cdot\,,1))
\end{aligned} 
\end{equation} 
in $]0,t_*[$ for each $k=1,\ldots,n$. The initial condition is obtained from the original problem by
\begin{equation}
 \label{kkkic} 
 \kkk(0,\,\cdot\,) = \kappa(0, \Phi_{a,b}^k (0,\,\cdot\,)) - f_{a,b}^k(0,\,\cdot\,) \quad \text{in } I. 
\end{equation}
We prove the existence of solutions to (\ref{kkkeqn}-\ref{kkkic}) by means of Banach fixed point theorem. Let us denote 
\begin{equation}
 \label{defX}
 X = \left\{
 \begin{aligned}
 \nn = &(\nn^1, \ldots, \nn^n) \in \left(C([0,t_*]; H^1_0(I)) \cap L^2(0,t_*;H^2(I))\right)^n \\ &\text{ such that }\nnk \text{ satisfies \eqref{kkkic} and }\\ &|\nn|_X^2 \leq 4 \max_{k = 1,\ldots, n} \|\kappa_{0,x}\|_{L^2(I^k)}^2 + \tfrac{19 \alpha^2}{m^4} 
\end{aligned}
\right\}, 
\end{equation} 
\begin{equation}
\label{defY}
 Y = \left\{
 \begin{aligned}
 (c,d) = &(c^1, \ldots, c^n, d^1, \ldots, d^n) \in H^1(0,t_*, \T)^{2n} \text{ such that }\\
 &\begin{array}{ll}
 c^k(0) = c^k_0,\; d^k(0) = d^k_0 &\text{ for } k=1,\ldots,n, \\ 
 \min |J^k| = \min|]c^k,d^k[| \geq m &\text{ for } k=1,\ldots,n , \\ 
 \min |G^k| = \min|]d^{k-1},c^k[| \geq m &\text{ for } k=1, \ldots, n,\\ 
 |(c, d)|_Y^2 \leq \tfrac{1}{3} . 
 \end{array} 
\end{aligned} 
\right\}.
\end{equation} 
Here, the number $m$ is chosen so that $|J^k_0| \geq 2m$ and $|G^k_0| \geq 2m$ for $k=1, \ldots, n$. We also introduced the notation 
\begin{equation} 
\label{defXn}
 |\nn|_X =\max_{k = 1,\ldots, n} \left(\int_{0}^{t_*}\!\!  \int_I (\nnk_{xx})^2 + \sup_{t \in [0,t_*]} \int_I (\nnk_x)^2 \right)^\frac{1}{2},
\end{equation}
\begin{equation}
\label{defYn}
 |(c, d)|_Y = \max_{k = 1,\ldots, n} \left(\|\dot{c}^k\|_{L^2(0,t_*)}^2 + \|\dot{d}^k\|_{L^2(0,t_*)}^2\right)^\frac{1}{2}  
\end{equation} 
for seminorms that induce metrics on $X$ and $Y$ and denoted $d^0 = d^n$. Further, we introduce operators $\R \colon Y \to X$ solving the system (\ref{kkkeqn}, \ref{kkkbc1}, \ref{kkkic}) for $\nn$ given $(c,d)$ and $S\colon X \to Y$ that solves the ODE system \eqref{kkkbc2} for $(c,d)$ given $\nn$. We will now show that these operators are well defined and that the composed operator
\begin{equation} 
 \label{RS}
 (\R \circ \S, \S \circ \R)\colon X \times Y \to X \times Y
\end{equation}
satisfies the assumptions of Banach fixed point theorem provided that $t_*$ is small enough. 

First we consider well-posedness of the operator $\R$. As $(c,d) \in H^1(0,t_*)^{2n}$, the problem of solving (\ref{kkkeqn}, \ref{kkkbc1}, \ref{kkkic}) is indeed well-posed in 
\[
\left(C([0,t_*], H^1_0(I)) \cap L^2(0,t_*,H^2(I))\right)^n 
\] 
and we have the following estimate on the solution $\nn$ 
\begin{multline} 
 \int_{0}^{t_*}\!\!  \int_I \left(\tfrac{\nnk_{xx}}{|J^k|}\right)^2 + 2\, \sup_{t \in [0,t_*]} \int_I (\nnk_x)^2 \leq 2 \int_I \nnk_x(0,\cdot\,)^2\\ 
 + 3 \left( \int_{0}^{t_*}\!\! \int_I ( \Phi^k_{c,d,t})^2 \left((\nnk_x)^2 + (f_{c,d,x}^k)^2\right) + \int_{0}^{t_*}\!\! \int_I |J^k|^2(f_{c,d,t}^k)^2 \right). 
\end{multline} 
Using inequalities 
\begin{equation*}
 \int_{0}^{t_*}\!\! \int_I ( \Phi^k_{c,d,t})^2 \left((\nnk_x)^2 + (f_{c,d,x}^k)^2\right) \leq |(c,d)|_Y^2 \left( \sup_{t \in [0,t_*]} \int_I (\nnk_x)^2 + \left( \tfrac{2 \alpha}{m}\right)^2\right) , 
\end{equation*} 
\begin{equation*} 
 \int_{0}^{t_*}\!\! \int_I |J^k|^2(f_{c,d,t}^k)^2 \leq \int_{0}^{t_*}\!\! \int_I (f_{c,d,t}^k)^2 \leq \tfrac{11 \alpha^2}{m^4} |(c,d)|_Y^2 ,
\end{equation*}
\begin{equation*} 
 \int_I \nnk_x(0,\cdot\,)^2 \leq 2 \|\kappa_{0,x}\|_{L^2(J^k)}^2  + 2 \tfrac{\alpha^2}{m^2}
\end{equation*}
and the definition of $Y$ we obtain $\nnk \in X$. 

Now, let $\nn \in X$. Due to parabolic trace embedding 
\begin{equation}
 \label{trace}
 C([0,t_*], L^2(I)) \cap L^2(0,t_*,H^1(I)) \hookrightarrow L^4(0,t_*, L^2(\partial I))
\end{equation} 
the problem of solving \eqref{kkkbc2} is locally well-posed and we have inequalities 
\begin{multline} 
\label{estabh1}
 |(c,d)|^2_Y \leq \tfrac{2}{\alpha^2 m^2} \max_{k = 1,\ldots,n}\|\nnk_x\|_{L^2(0,t_*,L^2(\partial I))}^2 + \tfrac{8}{m^2} t_* \\
 \leq  \tfrac{2}{\alpha^2 m^2} \max_{k = 1,\ldots,n} \|\nnk_x\|_{L^4(0,t_*,L^2(\partial I))}^2 t_*^{\frac{1}{2}} + \tfrac{8}{m^4} t_* \leq \tfrac{2 \gamma^2}{\alpha^2 m^2} |\nn|_X^2 t_*^{\frac{1}{2}} + \tfrac{8}{m^4} t_*, 
\end{multline}
\begin{multline} 
\label{estabc}
 \sup_{t \in [0, t_*]} |c^k(t) - c^k(0)| \leq \int_{0}^{t_*} |\dot{c}_k| \leq \tfrac{1}{\alpha m} \int_{0}^{t_*} |\nnk_x(\,\cdot\,,0) | + \tfrac{2}{m^2} {t_*} \\ \leq \tfrac{1}{\alpha m} \|\nnk_x(\,\cdot\,,0)\|_{L^4(0,t_*)} t_*^\frac{3}{4} + \tfrac{2}{m^2} t_* \leq \tfrac{\gamma}{\alpha m} |\nn|_X t_*^\frac{3}{4} + \tfrac{2}{m^2} t_*
\end{multline} 
for each $k = 1, \ldots, n$ and similarly with $d^k$, where $\gamma$ is the constant in the inequality  
\[ 
 \|u\|_{L^4(0,t_*,L^2(\partial I))} \leq \gamma \left(\int_{0}^{t_*}\!\!  \int_I (u_x)^2 + \sup_{t \in [0,t_*]} \int_I u^2 \right)^\frac{1}{2}
\]
connected to the embedding \eqref{trace}. Thus, $\S$ is well defined provided that $t_*$ is small enough. We also see that under this assumption $(\R \circ \S, \S \circ \R)$ maps $X\times Y$ into itself. We need yet to prove that this map is a contraction. 

First, let $\R (c, d) = \nn$ and $\R (c', d') = \nn'$. Let us denote $J'^k = ]c'^k, d'^k[$, $G'^k = ]d'^{k-1}, c'^k[$ for $k = 1, \ldots, n$, $d'^0 = d'^n$. Then, we have an inequality 
\begin{multline}
 \label{eww}
 \int_{0}^{t_*}\!\!  \int_I \left(\tfrac{\nnpk_{xx} - \nnk_{xx}}{|J^k|}\right)^2 + 2\,\sup_{t \in [0,t_*]} \int_I (\nnpk_x - \nnk_x)^2  \leq 5 \left(\int_{0}^{t_*}\!\!  \int_I \left(\tfrac{1}{|J^k|^2} - \tfrac{1}{|J'^k|^2}\right)^2 (\nnk_{xx})^2 \right.\\ 
 + \int_{0}^{t_*}\!\! \int_I \left(\Phi^k_{c,d,t} - \tfrac{|J^k|}{|J'^k|} \Phi^k_{c',d',t}\right)^2\left(\nnk_x + f^k_{c',d',x} \right)^2 + \int_{0}^T\!\! \int_I (\Phi^k_{c,d,t})^2 \left(\nnpk_x - \nnk_x \right)^2 \\ 
 + \left. \int_{0}^{t_*}\!\! \int_I (\Phi^k_{c,d,t})^2 \left(f^k_{c,d,x} - f^k_{c',d',x} \right)^2 + \int_{0}^{t_*}\!\! \int_I (f^k_{c,d,t} - f^k_{c',d',t})^2\right) . 
\end{multline} 
After a technical calculation involving application of embedding $H^1(0,t_*) \hookrightarrow C([0,t_*])$ to \eqref{eww} we 
obtain that $\R$ is Lipschitz continuous on $Y$. Now, if $\S \nn = (c,d)$ and 
$\S \nn' = (c',d')$, we can derive (again, owing to continuity of $H^1(0,t_*) 
\hookrightarrow C([0,t_*])$)
\begin{equation} 
 \label{eww2} 
 |(\dot{c} - \dot{c}', \dot{d} - \dot{d}')| \leq C(m) \left( \left|\left(c-c',d-d'\right)\right| + \|\nnk_x - \nnpk_x\|_{L^2(\partial I)^n}^2\right). 
\end{equation} 
from \eqref{freekbc2} ($|\cdot|$ denotes any norm on $\mathbb R^{2n}$). Invoking Gronwall's inequality and 
\begin{equation} 
 \label{h}
 \|h\|_{L^2(0,t_*;L^2(\partial I))}^2 \leq t_*^{\frac{1}{2}} \|h\|_{L^4(0,t_*,L^2(\partial I))}^2 \leq t_*^{\frac{1}{2}} |h|_X  
\end{equation}
we obtain that $\S$ is Lipschitz continuous with Lipschitz constant arbitrarily small for small $t_*$. Thus, choosing small enough $t_*$, we obtain existence of unique fixed point of \eqref{RS} which clearly solves (\ref{kkkeqn}-\ref{kkkic}). Inverting the rescaling \eqref{ktilde} yields a solution to the system of free boundary problems (\ref{freek}-\ref{freekbc2}) with initial datum $\kappa_0$ satisfying the assertions. 
\end{proof} 

We may extend $\kappa$ to each $F^k$ by suitable constants ($(-1)^k \frac{\alpha}{|F^k|}$). Resulting function belongs in fact to $C([0,t_*], H^1(\T))$.  
Finally, we construct the function $u$ solving (\ref{free}-\ref{freebc2}) with initial datum $u_0$ as
\[ 
 u(t,\cdot) = u_0 + \int_0^t \kappa .
\]
 
Using standard methods of linear parabolic regularity theory, we may extract from (\ref{kkkeqn}-\ref{kkkbc2}) further regularity of endpoint paths and the unfaceted part of solution. 

\begin{prop} 
 The solution $(u, (a,b))$ to (\ref{free}-\ref{freebc2}) constructed in Proposition \ref{freeexistence} belongs to 
 $$C^\infty\left(\bigcup_{k = 1, \ldots, n}\overline{I^k_{\delta,t_*}}\right) \times C^\infty([\delta, t_*])^{2n}$$ 
 for any $0 < \delta < t_*$. 
\end{prop}

\begin{proof}
 We perform a bootstrap procedure. First, we note that as the traces $\kkk(\cdot,0), \kkk(\cdot, 1)$ of solution constructed in Proposition \ref{freeexistence} belong to $L^4(0,t_*)$, also $(\dot a, \dot b) \in L^4(0,t_*)^{2n}$. Thus, cutting off the initial datum, we may solve (\ref{kkkeqn}, \ref{kkkbc1}) in 
 $$\left(W^{1,4}(\delta_0,t_*,L^4(I))\cap L^4(\delta_0,t_*,W^{2,4}(I))\right)^n$$ 
 with some $\delta_0 < \delta$ \cite[Chapter IV, Theorem 9.1]{lsu}. The traces $\kkk(\cdot,0), \kkk(\cdot, 1)$ of functions in this space (and therefore also $(\dot a, \dot b)$) belong to $W^{\frac{3}{8}, 4}(\delta_0, t_*)$ \cite[Chapter II, Lemma 3.4]{lsu} which in turn embeds continuously in $C^\frac{1}{8}([\delta_0, t_*])$ \cite[Theorem 8.2]{hitch}. Now we repeatedly apply \cite[Chapter IV, Theorem 5.2]{lsu}. Given coefficients and external force of \eqref{kkkeqn} in parabolic H\" older class $C^{\frac{k}{2} + \frac{1}{8}, k + \frac{1}{4}}([\delta_k,t_*]\times I)$, $k=0, 1,\ldots$, we yield the solution in 
 $$C^{\frac{k+2}{2} + \frac{1}{8}, k+2 + \frac{1}{4}}([\delta_{k+1},t_*]\times I)$$ 
 with $\delta_{k} < \delta_{k+1} < \delta$. The traces $\kkk(\cdot,0), \kkk(\cdot, 1)$ of any element of this space belong to $C^{\frac{k+1}{2} + \frac{1}{8}}([\delta_{k+1},t_*])$ (see \cite[Exercise 8.8.6]{krylovh}) which raises the regularity of coefficients and force of \eqref{kkkeqn} one step and allows the procedure to continue. The $C^\infty$ regularity is preserved in the passage to the solution to (\ref{free}-\ref{freebc2}). 
\end{proof}

 \section*{Acknowledgements} 
 The author wishes to express his gratitude towards Piotr Mucha for suggesting the problem, encouragement and helpful discussions. Special thanks are also due to Jose Maz{\'o}n for his informative comments. 
 
 The work has been supported by the grant no.\;2014/13/N/ST1/02622 of the National Science Centre, Poland.

 \bibliographystyle{plain}
 \bibliography{aniso}
\end{document}